%% file: arxiv.tex
\documentclass[letterpaper,11pt,reqno]{amsart}

\usepackage[margin=1in]{geometry}

\usepackage{hyperref}
\hypersetup{
     colorlinks   = true,
     citecolor    = black,
     linkcolor    = blue,
     urlcolor     = black
}

\usepackage{graphicx,amsmath,amssymb,amsthm,paralist,color,tikz-cd}
\usepackage{mathrsfs}
\usepackage{mathtools}
\usepackage{setspace}

\usepackage[utf8]{inputenc} 
\usepackage[T1]{fontenc}

\usepackage{amscd}

\usepackage{bbm}

\input{macros}

\title[Flattening and Fourier Decay]{Polynomial Fourier Decay For Patterson-Sullivan Measures}
\author{Osama Khalil}
\address{Department of Mathematics, Statistics, and Computer Science, University of Illinois Chicago, Chicago, IL}
\email{okhalil@uic.edu}

\date{}

\begin{document}

\begin{abstract}
    We show that the Fourier transform of Patterson-Sullivan measures associated to convex cocompact groups of isometries of real hyperbolic space decays polynomially quickly at infinity. 
    The proof is based on the $L^2$-flattening theorem obtained in~\cite{Khalil-Mixing} combined with a method based on dynamical self-similarity for ruling out the sparse set of potential frequencies where the Fourier transform can be large. 
\end{abstract}

\maketitle

\section{Introduction}

\subsection{Background}

The Fourier transform of a Borel probability measure $\mu$ on $\R^d$ is defined as follows:
\begin{align}
    \hat{\mu}(\xi) := \int_{\R^d} e^{2\pi i\langle \xi, x\rangle} \;d\mu(x), \qquad \xi\in\R^d.
\end{align}
We say $\mu$ has \textit{polynomial Fourier decay} if $|\hat{\mu}(\xi)| = O(\norm{\xi}^{-\k})$ for some $\k>0$ as $\norm{\xi}\to \infty$.

Rates of decay of Fourier transforms of dynamically defined measures have been extensively studied in recent years.
Beyond its intrinsic interest, this question has found many applications in other areas of mathematics; e.g.~essential spectral gaps on hyperbolic manifolds~\cite{DyatlovZahl,BourgainDyatlov,LiNaudPan}, the uniqueness problem~\cite{LiSahlsten}, quantum chaos and fractal uncertainty principles~\cite{Dyatlov-IntroFUP}, Diophantine approximation~\cite{DavenportErdosLeVeque}, and geometric measure theory~\cite{Shmerkin-AbsContBernoulliConv,Mattila-FourierBook} to name a few.

Moreover, the problem has motivated the development of many methods drawing on a wide varying tools ranging from spectral gaps of the underlying dynamics~\cite{AlgomRHWang-Polynomial,BakerSahlsten}, to renewal theory~\cite{Li-Stationary}, sum-product phenomena~\cite{BourgainDyatlov,LiNaudPan,Leclerc-JuliaSets}, large deviation estimates for Fourier transforms~\cite{MosqueraShmerkin,AlgomChangWuWu,BakerBanaji}, as well as many related developments; cf.~\cite{AlgomHertzWang-Normality,AlgomHertzWang-Log,Sahlsten-Gauss,SahlstenStevens,LiSahlsten,LiSahlsten-SelfAffine,VarjuYu,Bremont} for a non-exhaustive list.
We refer the reader to the survey~\cite{Sahlsten-survey} for a comprehensive account of the history and recent developments in the subject.

A common strategy that is implicit in many of the aforementioned results proceeds as follows:
\begin{enumerate}
    \item\label{item:flattening step} Find a mechanism to show that the Fourier transform has the desired rate of decay for a large set of frequencies $\xi$.
    
    \item\label{item:Step 2} Use the dynamics (or the multiscale/convolution structure of $\mu$) to express the Fourier transform of $\mu$ at frequency $\xi$ as an \textit{average} of Fourier transforms of (scaled copies of) $\mu$ at images of $\xi$ by the dynamics.

    \item\label{item:Step 3} Show (through non-linearity of/Diophantine conditions on the dynamics) that images of $\xi$ by the dynamics are reasonably well-distributed in the space in such a way that they avoid the potential exceptional set of frequencies arising in Step~\ref{item:flattening step}.
\end{enumerate}

To demonstrate this strategy, consider the following basic estimate towards Step~\ref{item:flattening step}: if $\mu$ satisfies the Frostman condition $\mu(B(x,r)) \lesssim r^\a $ for some $\a>0$ and all balls of radius $r\geq 0$, then the Fourier transform decays like $\norm{\xi}^{-\a/2}$ \emph{on average}, i.e.
\begin{align}\label{eq:Frostman}
    \int_{\norm{\xi}\leq R} |\hat{\mu}(\xi)|^2\;d\xi \lesssim R^{d-\a}, \qquad \forall R\geq 1.
\end{align}

This estimate roughly means that the exceptional set of potentially problematic frequencies have (box) dimension at most $d-\a$; cf.~\cite[Section 3.8]{Mattila-FourierBook}.
Hence, we can obtain Fourier decay as soon as we can show that the frequencies produced in Step~\ref{item:Step 2} have dimension $>d-\a$.
Since the image of $\xi$ under the dynamics tends to have a similar dimension to the support of $\mu$ itself, this procedure is sufficient for establishing Fourier decay in many situations when the dimension of the support of $\mu$ is $>d/2$.

However, in general, the estimate~\eqref{eq:Frostman} is rather weak when $\a\leq d/2$. 
In that case, more involved methods are necessary to either produce stronger estimates in Step~\ref{item:flattening step} (e.g.~large deviations methods) or to produce better averaging and well-distribution schemes in Steps~\ref{item:Step 2} and~\ref{item:Step 3} (e.g.~spectral gap and renewal theory methods).

Recently, a very general estimate towards Step~\ref{item:flattening step} was obtained in~\cite[Corollary 11.5]{Khalil-Mixing} under natural non-concentration hypotheses on the (not necessarily dynamically defined) measure $\mu$. 
Namely, it is shown\footnote{Cf.~Definition~\ref{def:non-conc} and Theorem~\ref{thm:flattening} for precise statements.} that if $\mu$ does not concentrate near proper affine subsapces of $\R^d$ at many scales, then its Fourier transform decays polynomially outside of a very sparse set of frequencies, i.e.~for all $\e>0$, there is $\d>0$ such that:
\begin{align}\label{eq:flattening intro}
    \left|\set{\norm{\xi}\leq R: |\hat{\mu}(\xi)|>R^{-\d}}\right| = O\left(R^{\e}\right).
\end{align}

The goal of this article to show that~\eqref{eq:flattening intro} can be used in conjunction with the strategy outlined above to give efficient proofs of quantitative Fourier decay of dynamically defined measures.
We apply our method to a particular class of interest in applications, namely that of Patterson-Sullivan measures for convex cocompact groups of isometries of real hyperbolic space.
Since the non-concentration conditions implying~\eqref{eq:flattening intro} are known to hold for large classes of dynamically-defined measures\footnote{E.g.~self-conformal measures~\cite{Dasetal}, and Patterson-Sullivan measures for (cusped) geometrically finite manifolds~\cite{Khalil-Mixing}.}, we hope the simplicity of the method presented here will allow it to be extended to yield Fourier decay results in much broader contexts.

\subsection{Main result}

Let $\G$ be a discrete, Zariski-dense, convex cocompact, group of isometries of real hyperbolic space $\H^{d+1}$, $d\geq 1$.
Let $\L_\G$ be the limit set of $\G$ on $\partial \H^{d+1}$ and $\mu$ be the Patterson-Sullivan probability measure on $\L_\G$ associated to $\G$; cf.~Section~\ref{sec:prelims} for detailed definitions.
The following is the main result of this article.

\begin{thm}\label{thm:decay intro}
There exists $\k>0$ such that the following holds for all $\vp\in C^2$, $\psi\in C^1$ satisfying
\begin{align*}
    \norm{\vp}_{C^2}+ \norm{\psi}_{C^1} \leq A,  \qquad \inf_{x\in \L_\G} \norm{\nabla_x\vp}> a,
\end{align*}
for some constants $a>0$ and $A\geq 1$. 
There exists a constant $C=C(A,a,\mu)\geq 1$, so that for all $\l\neq 0$, we have
\begin{align*}
    \left|\int_{\L_\G} e^{2\pi i\l \vp(x)} \psi(x) \;d\mu(x)\right| \leq C |\l|^{-\k}.
\end{align*}
\end{thm}

\begin{remark}\label{rem:covers}
    Our proof shows that the rate $\k$ provided by Theorem~\ref{thm:decay intro} depends only on non-concentration parameters of $\mu$; cf.~Section~\ref{sec:flatenning} for the precise definition of non-concentration. In particular, the rate of decay does not change upon replacing $\G$ by a finite index subgroup since the measure $\mu$ remains the same in this case~\cite{Roblin-2}.
\end{remark}

Theorem~\ref{thm:decay intro} generalizes prior work of Bourgain and Dyatlov in the case of hyperbolic surfaces \cite{BourgainDyatlov} and of Li, Naud, and Pan in the case of Schottky hyperbolic $3$-manifolds \cite{LiNaudPan}.
These prior results are based on Bourgain's sum-product theorem, while the proof of Theorem~\ref{thm:decay intro} is based on the estimate~\eqref{eq:flattening intro}, which was obtained using purely additive methods.

To keep the presentation clear, we restricted our setup to the case of convex cocompact groups.
Using the recurrence results obtained in~\cite{Khalil-Mixing}, the proof of Theorem~\ref{thm:decay intro} can be adapted to handle the general case of geometrically finite manifolds.

By the work of Dyatlov and Zahl~\cite{DyatlovZahl}, Theorem~\ref{thm:decay intro} is known to imply spectral bounds on the resolvent of the Laplace operator which yield an \textit{essential spectral gap}\footnote{That is to say a strip to the left of the critical line with at most finitely many poles. The interested reader is referred to the survey~\cite{Dyatlov-IntroFUP} for more on this topic.} for the resolvent as well as Selberg's zeta function. Moreover, the size of the essential spectral gap obtained this way depends explicitly on the decay rate $\k$ in Theorem~\ref{thm:decay intro}. 
In particular, Theorem~\ref{thm:decay intro} implies that the resolvent admits a uniform essential spectral gap over all finite covers of $\H^{d+1}/\G$; cf.~Remark~\ref{rem:covers}.
We note that this essential spectral gap result was obtained independently in~\cite{BackusLengTao} by different methods and an essential gap of size depending on $\G$ was obtained previously in~\cite{Naud-Cantor,PetkovStoyanov}.

\begin{acknowledgement}
    This author is partially supported under NSF grant DMS-2247713. 
\end{acknowledgement}

\section{Preliminaries}

    \label{sec:prelims}

	\subsection{Convex cocompact manifolds}

	The standard reference for the material in this section is~\cite{Bowditch1993}.
	Let $G $ denote the group of orientation preserving isometries of real hyperbolic space, denoted $\H^{d+1}$, of dimension $d\geq 1$.
    In particular, $G\cong \mrm{SO}(d+1,1)^0$.
	
    Fix a basepoint $o\in \H^{d+1}$. Then, $G$ acts transitively on $\H^{d+1}$ and the stabilizer $K$ of $o$ is a maximal compact subgroup of $G$.
    We shall identify $\H^{d+1}$ with $K\backslash G$.
    Denote by $A=\set{g_t:t\in\R}$ a one parameter subgroup of $G$ inducing the geodesic flow on the unit tangent bundle of $\H^{d+1}$.
    Let $M<K$ denote the centralizer of $A$ inside $K$ so that the unit tangent bundle $\uT\H^{d+1}$ may be identified with ${M}\backslash{G}$.
    In Hopf coordinates, we can identify $\uT\H^{d+1}$ with $\R\times (\partial\H^{d+1}\times \partial\H^{d+1} \setminus \Delta)$, where $\partial \H^{d+1}$ denotes the boundary at infinity and $\Delta$ denotes the diagonal.
    
    Let $\G<G$ be an infinite discrete subgroup of $G$.
    The limit set of $\G$, denoted $\L_\G$, is the set of limit points of the orbit $\G\cdot o$ on $\partial \H^{d+1}$.
    Note that the discreteness of $\G$ implies that all such limit points belong to the boundary.
    Moreover, this definition is independent of the choice of $o$ in view of the negative curvature of $\H^{d+1}$.
    We often use $\L$ to denote $\L_\G$ when $\G$ is understood from context.
    We say $\G$ is \textit{non-elementary} if $\L_\G$ is infinite.

    The \textit{non-wandering set} for the geodesic flow is the closure of the set of vectors in the unit tangent bundle whose orbit accumulates on itself.
    In Hopf coordinates, this set, denoted $\Omega$, coincides with the projection of $\R\times (\L_\G\times\L_\G-\Delta)$ mod $\G$.
	We say $\H^{d+1}/\G$ is \textit{convex cocompact} if $\Omega$ is compact, cf.~\cite{Bowditch1993}.
    Denote by $N^+$ the expanding horospherical subgroup of $G$ associated to $g_t$, $t\geq 0$.

   Given $g\in G$, we denote by $g^+$ the coset of $P^-g$ in the quotient $P^-\backslash G$, where $P^-=N^-AM$ is the stable parabolic group associated to $\set{g_t:t\geq 0}$.
   Similarly, $g^-$ denotes the coset $P^+g$ in $P^+\backslash G$.
   Since $M$ is contained in $P^\pm$, such a definition makes sense for vectors in the unit tangent bundle $M\backslash G$.
   Geometrically, for $v\in M\backslash G$, $v^+$ (resp.~$v^-$) is the forward (resp.~backward) endpoint of the geodesic determined by $v$ on the boundary of $\H^{d+1}$.
   Given $x\in G/\G$, we say $x^{\pm}$ belongs to $\L$ if the same holds for any representative of $x$ in $G$; this notion being well-defined since $\L$ is $\G$ invariant.

	\subsection*{Notation} Throughout the remainder of the article, we fix a discrete, Zariski-dense, convex cocompact group $\G$ of isometries of $\H^{d+1}$.

\subsection{Patterson-Sullivan measures}

    The \textit{critical exponent}, denoted $\d_\G$, is defined to be the infimum over all real number $s\geq 0$ such that the Poincar\'e series
    \begin{align}\label{eq:Poincare}
        P_\G(s,o) := \sum_{\g\in\G} e^{-s d(o,\g\cdot o)}
    \end{align}
    converges.
    This exponent coincides with the Hausdorff dimension of the limit set as well as the topological entropy of the geodesic flow on the quotient orbifold $\H^{d+1}/\G$.
    We shall simply write $\d$ for $\d_\G$ when $\G$ is understood from context.
    
    The \textit{Busemann function} is defined as follows: given $x,y\in \H^{d+1}$ and $\xi\in \partial \H^{d+1}$, let $\g:[0,\infty)\to\H^{d+1}$ denote a geodesic ray terminating at $\xi$ and define
    \begin{equation*}
        \b_\xi(x,y) = \lim_{t\to\infty}
        \dist (x,\g(t)) - \dist(y,\g(t)).
    \end{equation*}
    A $\G$-invariant conformal density of dimension $s$ is a collection of Radon measures $\set{\nu_x}$ on the boundary indexed by $ x\in \H^{d+1}$ which satisfy the following equivariance property:
    \begin{equation*}
        \g_\ast \nu_x = \nu_{\g x}, \qquad \text{and} \qquad
        \frac{d\nu_{y}}{d\nu_x}(\xi) = e^{-s\b_{\xi}(x,y)}, \qquad
        \forall x,y\in\H^{d+1}, \xi\in \partial \H^{d+1}, \g\in\G.
    \end{equation*}

    Patterson~\cite{Patterson} and Sullivan~\cite{Sullivan} showed the existence of a unique (up to scaling) $\G$-invariant conformal density of dimension $\d_\G$, denoted $\set{\ps_x:x\in \H^{d+1}}$.
    These measures are known as the \textit{Patterson-Sullivan measures}.
    We refer the reader to~\cite{Roblin} and~\cite{PaulinPollicottSchapira} and references therein for details of the construction in much greater generality.

\subsection{Stable and unstable foliations and leafwise measures}

    Recall that we fixed a basepoint $o\in \H^{d+1}$.
    In what follows, we use the following notation for pullbacks of the Patterson-Sullivan measures to orbits of $N^+$ under the visual map:
    \begin{equation}\label{eq:unstable conditionals}
        d\mu_x^u(n) = e^{\d_\G \b_{(nx)^+}(o,nx)}d\ps_o((nx)^+).
    \end{equation}
    These measures have simpler transformation formulas under the action of the geodesic flow and $N^+$ which makes them relatively easier to analyze than the Patterson-Sullivan measures directly.
    In particular, they satisfy the following equivariance property under the geodesic flow:
    \begin{equation}\label{eq:g_t equivariance}
        \mu_{g_tx}^u = e^{\d t} \mrm{Ad}(g_t)_\ast \mu_{x}^u.
    \end{equation}
    Moreover, it follows readily from the definitions that for all $n\in N^+$,
    \begin{align}\label{eq:N equivariance}
       (n)_\ast \mu_{nx}^u =  \mu_x^u,
    \end{align}
    where $(n)_\ast \mu_{nz}^u$ is the pushforward of $\mu_{nz}^u$ under the map $u\mapsto un$ from $N^+$ to itself.
    Finally, since $M$ normalizes $N^+$, these conditionals are $\Ad(M)$-invariant in the sense that for all $m\in M$,
    \begin{align}\label{eq:M equivariance}
        \mu^u_{mx}  = \Ad(m)_\ast\mu_x^u.
    \end{align}
    
\subsection{Local stable holonomy}\label{sec:holonomy}

    In this Section, we recall the definition of (stable) holonomy maps which are essential for our arguments. We give a simplified discussion of this topic which is sufficient in our homogeneous setting homogeneous.
    Let $x=u^-y$ for some $y\in \Omega$ and $u^-\in N^-_2$.
Since the product map $N^-\times A \times M \times N^+\to G$ is a diffeomorphism near identity, we can choose the norm on the Lie algebra so that the following holds. We can find maps $p^-:N_1^+\to P^-=N^-AM$ and $u^+:N_2^+\to N^+$ so that
\begin{align}\label{eq:switching order of N- and N+}
    nu^- = p^-(n)u^+(n), \qquad \forall n\in N_2^+.
\end{align}
Then, it follows by~\eqref{eq:unstable conditionals} that for all $n\in N_2^+$, we have
\begin{align*}
    d\mu_y^u(u^+(n)) = e^{\d \b_{(nx)^+}(u^+(n)y,nx)}d\mu_x^u(n).
\end{align*}
Moreover, by further scaling the metrics if necessary, we can ensure that these maps are diffeomorphisms onto their images.
In particular, writing $\Phi(nx)=u^+(n)y$, we obtain the following change of variables formula: for all $f\in C(N_2^+)$,
\begin{align}\label{eq:stable equivariance}
    \int f(n) \;d\mu_x^u(n) = 
    \int f((u^+)^{-1}(n)) e^{-\d \b_{\Phi^{-1}(ny)}(ny,\Phi^{-1}(ny))}\;d\mu_y^u(n).
\end{align}

\begin{remark}\label{rem:commutation of stable and unstable}
To avoid cluttering the notation with auxiliary constants, we shall assume that the $N^-$ component of $p^-(n)$ belongs to $N_2^-$ for all $n\in N_2^+$ whenever $u^-$ belongs to $N_1^-$.
\end{remark}


\subsection{Notational convention} Throughout the article, given two quantities $A$ and $B$, we use the Vinogradov notation $A\ll B$ to mean that there exists a constant $C\geq 1$, possibly depending on $\G$ and the dimension of $G$, such that $|A|\leq C B$. In particular, this dependence on $\G$ is suppressed in all of our implicit constants, except when we wish to emphasize it.
    The dependence on $\G$ may include for instance the diameter of the complement of our choice of cusp neighborhoods inside $\Omega$ and the volume of the unit neighborhood of $\Omega$.
    We write $A\ll_{x,y} B$ to indicate that the implicit constant depends parameters $x$ and $y$.
    We also write $A=O_x(B)$ to mean $A\ll_x B$.


\subsection{The $L^2$-flattening theorem}
\label{sec:flatenning}

In light of the formula~\eqref{eq:unstable conditionals}, Theorem~\ref{thm:decay intro} amounts to studying the Fourier transform of the measures $\mu_x^u$. Moreover, the isomorphism $N^+\cong \R^d$ allows us to view these measures as living on Euclidean space.

The key ingredient in the proof of Theorem~\ref{thm:decay intro} is~\cite[Corollary 1.8]{Khalil-Mixing} which relates Fourier decay properties of measures on Euclidean space to the non-concentration properties of such measures near proper affine subspaces. 
This result in particular implies that PS measures enjoy polynomial Fourier decay outside of a very sparse set of frequencies.

We formulate here a special case of the aforementioned result which suffices for convex cocompact manifolds and refer the reader to~\cite[Theorem 11.5]{Khalil-Mixing} for a more general result that holds in the presence of cusps.

\begin{definition}\label{def:non-conc}
We say that Borel measure $\mu$ on $\R^d$ is \textit{uniformly affinely non-concentrated} if for every $\e>0$, there exists $\d(\e)>0$ so that $\d(\e)\to 0$ as $\e\to 0$ and for all $x\in \R^d$, $0<r\leq 1$, and every affine hyperplane $W<\R^d$, we have

        \begin{align}\label{eq:uniform affine non-conc}
            \mu( W^{(\e r)} \cap B(x,r)) \leq \d(\e) \mu(B(x,r)),
        \end{align}
        where $W^{(r)}$ and $B(x,r)$ denote the $r$-neighborhood of $V$ and the $r$-ball around $x$ respectively.
        We refer to $\d(\e)$ as \textit{the non-concentration parameters} of $\mu$.
\end{definition}

\begin{thm}[{\cite[Corollary 1.8]{Khalil-Mixing}}]
\label{thm:flattening}
Let $\mu$ be a compactly supported Borel probability measure on $\R^d$ which is uniformly affinely non-concentrated and denote by $\hat{\mu}$ its Fourier transform.
Then, for every $\e>0$, there is $\d>0$ such that for all $T\geq 1$,
\begin{align*}
    \left|\set{\norm{\xi}\leq T: |\hat{\mu}(\xi)|>T^{-\d}}\right| = O_\e (T^\e),
\end{align*}
where $|\cdot|$ denotes the Lebesgue measure on $\R^d$.
The implicit constant depends only on the non-concentration parameters of $\mu$ and the diameter of its support.
\end{thm}

We note that Theorem~\ref{thm:flattening} was obtained by different methods for measures on the real line in~\cite{RossiShmerkin}.

We will be able to apply Theorem~\ref{thm:decay intro} to PS measures (or, more precisely, their shadows $\mu_x^u$) thanks to the following proposition.

\begin{prop}[{\cite[Corollary 12.2]{Khalil-Mixing}}]
\label{prop:aff non-conc of PS}
    For every $x\in N_1^-\Omega$, the measure $\mu_x^u\left|_{N_1^+}\right.$ is uniformly affinely non-concentrated in the sense of Definition~\ref{def:non-conc}, with uniform parameters in $x$.
\end{prop}

\begin{remark}
    It is shown in~\cite[Corollary 12.2]{Khalil-Mixing} that the non-concentration parameters of $\mu_x^u$ depend only on the injectivity radius at $x$, which is in turn uniformly bounded above and below on a neighborhood of the non-wandering set $\Omega$ due to convex cocompactness of $\G$.
\end{remark}

\section{Proof of Theorem~\ref{thm:decay intro}}

The goal of this section is to provide the proof of Theorem~\ref{thm:decay intro}.
In light of the formula~\eqref{eq:unstable conditionals}, it suffices to prove polynomial Fourier decay for the measures $\mu_x^u$ for $x$ in the non-wandering set $\Omega$.

\subsection*{Norms and Lie algebras}

    In what follows, we denote by $\mf{n}^+$ and $\mf{n}^-$ the Lie algebras of $N^+$ and $N^-$ respectively. 
    We fix an isomorphism of $\mf{n}^+$ and $\mf{n}^-$ using a Cartan involution sending $g_{t}$ to $g_{-t}$. 
    Moreover, we fix an isomorphism of $\mf{n}^+$ (and hence of $\mf{n}^-$) with $\R^d$.
    Finally, we fix a Euclidean inner product on $\R^d\cong\mf{n}^+\cong \mf{n}^-$ denoted with $\langle \cdot,\cdot\rangle$ which is invariant by the Adjoint action of the group $M\cong \mrm{SO}_d(\R)$.

\subsection*{Reduction to linear phases}
We begin with the following elementary lemma which reduces the proof to the study of linear phase functions. Its proof based on using a partition of unity on the support of the integral and Taylor expanding the phase function. The details are left to the reader.

\begin{lem}\label{lem:reduce to linear}
To prove Theorem~\ref{thm:decay intro}, it suffices to show that there exists $\k>0$ so that for all $0\neq \xi\in  \R^{d}$, $x\in \Omega$, and $\psi\in C_c^1(N_1^+)$, we have
\begin{align}\label{eq:linear decay}
    \int_{N_1^+} e^{ i \langle \xi, n\rangle} \psi(n)\;d\mu^u_x(n)
    \ll_\G \norm{\psi}_{C^1} \norm{\xi}^{-\k},
\end{align}
where, by abuse of notation, if $n=\exp(v)$ for some $v\in \mf{n}^+\cong \R^d$, we let $\langle \xi,n\rangle := \langle \xi, v\rangle$.
\end{lem}

In the remainder of this section, we fix $\xi\in \R^d$ and $\psi\in C_c^1(N_1^+)$. Our goal is to prove the estimate~\eqref{eq:linear decay}.

\subsection*{Partitions of unity and flow boxes}

Let $\iota$ denote the smaller of $1$ and the injectivity radius of $\Omega$ and set
\begin{align}\label{eq:iota_xi}
    \iota_\xi := \iota/\norm{\xi}^{1/3}.
\end{align}

We let $\Pcal_\xi$ denote a partition of unity of the unit neighborhood of $\Omega$ so that each $\rho\in \Pcal_\xi$ is $M$-invariant and supported inside a flow box $B_\rho$ of radius\footnote{That is the support of the projection of the $M$-invariant function $\rho$ is contained in such a flow box.} $\iota_\xi$.
With the aid of the Vitali covering lemma, we can arrange for the collection $\set{B_\rho}$ to have a uniformly bounded multiplicity, depending only on the dimension of $G$.
We can choose such a partition of unity so that for all $\rho\in \Pcal_\xi$,
\begin{equation}\label{eq:norm of rho}
    \norm{\rho}_{C^1} \ll \iota_\xi^{-1}.
\end{equation}

We also introduce the following subcollection of $\Pcal_\xi$:
\begin{align}\label{eq:subpartition of unity near omega}
    \Pcal_\xi^0 := \set{\rho\in \Pcal_\xi: B_\rho\cap N_{1/2}^-\Omega \neq \emptyset}.
\end{align} 
Note that the cardinality of $\Pcal_\xi^0$ can be bounded as follows.
Indeed, since $\G$ is geometrically finite, the unit neighborhood of $\Omega$ has finite volume. Moreover, the flow boxes $B_\rho$ with $\rho \in \Pcal^0_\xi$ are all contained in such a unit neighborhood and have uniformly bounded multiplicity; cf.~\eqref{eq:subpartition of unity near omega}. 
Finally, each $B_\rho$ has radius $\iota_\xi$ for all $\rho\in\Pcal_\xi$.
Thus, letting $D\in\N$ be such that the Lebesgue measure of $B_\rho$ is $\asymp \iota_\xi^D$, we see that
\begin{equation}\label{eq:number of flow boxes}
    \#\Pcal_\xi^0 \ll_\G \iota_\xi^{-(2D+1)}.
\end{equation}
Note that the dimension of $X$ is $2D+1+\dim(M)$, however the bound above involves $2D+1$ only since each flow box is $M$-invariant.

 \subsection*{Transversals}
We fix a system of transversals $\{T_\rho\}$ to the strong unstable foliation inside the boxes $B_\rho$.
Since $B_\rho$ meets $N_{1/2}^-\Omega$ for all $\rho\in \Pcal_\xi^0$, we fix some $y_\rho$ in the intersection $B_\rho\cap N_{1/2}^-\Omega$. 
In this notation, we can find neighborhoods of identity $P^-_\rho \subset P^-= MAN^-$ and $N^+_\rho \subset N^+$ such that
\begin{equation}\label{eq:box notation}
    B_\rho = N^+_\rho P^-_\rho \cdot y_\rho, \qquad T_\rho = P^-_\rho \cdot y_\rho.
\end{equation}
We also let $M_\rho, A_\rho$, and $N^-_\rho$ be neighborhoods of identity in $M, A$ and $N^-$ respectively so that $P^-_\rho=M_\rho A_\rho N^-_\rho$.

\subsection*{Time smoothing}
Let $T>0$ be a large parameter to be chosen depending on $\norm{\xi}$.
Let $q:\R\to [0,1]$ be a smooth bump function supported in the interval $(T-1,T+1)$ and having the property
\begin{align}\label{eq:q(t)}
    \int_\R q(t)\;dt=1.
\end{align}
Using our partition of unity, we can write 
\begin{align}\label{eq:initial partition}
    \int_{N_1^+} e^{ i \langle \xi, n\rangle} \psi(n)\;d\mu^u_x(n)
    = \int_\R q(t)
    \sum_{\rho\in \Pcal_\xi} \int_{N_1^+}e^{i\langle \xi,n\rangle} \psi(n) \rho(g_{t} nx) &\;d\mu_x^u(n)dt.
\end{align}

\subsection*{Saturation and post-localization}
Fix some arbitrary $t>0$.
Our first step is to partition the integral in~\eqref{eq:initial partition} over $N_1^+$ into pieces according to the flow box they land in under flowing by $g_t$.
To simplify notation, we write
\begin{equation}\label{eq:y_t}
      x_t:= g_t x.
\end{equation}
We denote by $N_1^+(t)$ a neighborhood of $N_1^+$ defined by the property that the intersection 
$$ B_\rho\cap (\mrm{Ad}(g_t)(N_1^+(t))\cdot x_t)$$
consists entirely of full local strong unstable leaves in $B_\rho$.
We note that since $\mrm{Ad}(g_t)$ expands $N^+$ and $B_\rho$ has radius $<1$, $N_1^+(t)$ is contained inside $N_2^+$.
Since $\phi$ is supported inside $N_1^+$, we have
\begin{equation}\label{eq:N_1^+(t)}
    \chi_{N_1^+}(n)\psi(n)  = \chi_{N_1^+(t)}(n)\psi(n), \qquad \forall n \in N^+.
\end{equation}
For simplicity, we set
\begin{equation*}
    \xi_t := e^{-t}\xi, \qquad 
    \psi_t(n) := \psi(\Ad(g_t)^{-1} n ), \qquad \Acal_t :=  \mrm{Ad}(g_t)(N_1^+(t)).
\end{equation*}
For $\rho\in \Pcal_\xi$, we let $\Wcal_{\rho,t}$ denote the collection of connected components of the set
\begin{equation*}
    \set{n\in\Acal_t: nx_t\in B_\rho}.
\end{equation*}

Moreover, since $x\in \Omega$, we see that the the restriction of the support of $\mu_{x}^u$ to $N_1^+$ consists of points $n\in N^+$ with $nx \in \Omega$.
This implies that the non-zero summands in the right side of~\eqref{eq:initial partition} necessarily correspond to those $\rho$ in $\Pcal_\xi^0$.

In view of~\eqref{eq:N_1^+(t)}, changing variables using~\eqref{eq:g_t equivariance} yields
\begin{align}\label{eq:localize space}
  \sum_{\rho\in \Pcal_\xi} \int_{N_1^+}e^{i\langle \xi,n\rangle} \psi(n) \rho(g_{s+t} nx) &\;d\mu_x^u(n)
 \nonumber\\
&=e^{-\d t } 
    \sum_{\rho\in \Pcal_\xi^0, W\in \Wcal_{\rho,t}}
    \int_{n\in W}  e^{i\langle \xi_t,n\rangle} \psi_t(n) 
        \rho(nx_t)  \; d\mu_{x_t}^u(n).
\end{align}

\subsection*{Centering the integrals}
It will be convenient to center all the integrals in~\eqref{eq:localize space} so that their basepoints belong to the transversals $T_\rho$ of the respective flow box $B_\rho$; cf.~\eqref{eq:box notation}.

Let $I_{\rho,t}$ denote an index set for $\Wcal_{\rho,t}$.
For $W\in \Wcal_{\rho,t}$ with index $\ell\in I_{\rho,t}$, let $n_{\rho,\ell}\in W$, $m_{\rho,\ell}\in M_\rho$, $n_{\rho,\ell}^-\in N^-_\rho$, and $t_{\rho,\ell}\in(-\iota_\xi,\iota_\xi)$ be such that
\begin{equation}\label{eq:centers}
    x_{\rho,\ell} := n_{\rho,\ell}\cdot x_t = n_{\rho,\ell}^- m_{\rho,\ell}g_{t_{\rho,\ell}}  \cdot y_\rho \in T_\rho.
\end{equation}
Note that since $x$ belongs to $\Omega$ and $y_\rho\in N_{1/2}^-\Omega$, we have that
\begin{equation}\label{eq:x_rho,ell in omega}
    x_{\rho,\ell}\in \Omega.
\end{equation}

For each such $\ell$ and $W$, let us denote $W_\ell=Wn_{\rho,\ell}^{-1}$ and set
\begin{equation}\label{eq:phi tilde}
    \widetilde{\psi}_{\rho,\ell}(t,n) :=
          \psi_t( n n_{\rho,\ell}),
          \qquad 
          \tilde{\chi}_{\rho,\ell}(t,n) := 
          \exp(i \langle\xi_t, nn_{\rho,\ell} \rangle).
\end{equation} 
Changing variables using~\eqref{eq:g_t equivariance} and~\eqref{eq:N equivariance}, we can rewrite the right side of~\eqref{eq:localize space} as follows: 
\begin{align}\label{eq:center integrals on transversal}
         e^{-\d t} 
         \sum_{\rho\in \Pcal^0_\xi, W\in \Wcal_{\rho,t}}
          &\int_{n\in W}  e^{i\langle \xi_t,n\rangle}
          \psi_t( n  )
         \rho(nx_t)  \;d\mu_{x_t}^u(n) 
         \nonumber\\
         &=e^{-\d t}
         \sum_{\rho\in \Pcal_\xi^0} 
         \sum_{\ell\in I_{\rho,t}} 
           \int_{n\in W_\ell}  
           \widetilde{\chi}_{\rho,\ell}(t,n)
          \widetilde{\psi}_{\rho,\ell}(t,n) \rho(nx_{\rho,\ell})
            \;d\mu_{x_{\rho,\ell}}^u(n) .
\end{align}

\subsection*{Weak-stable holonomy} Fix some $\rho\in \Pcal_\xi^0$.
Recall the points $y_{\rho}\in T_\rho$ and $n^-_{\rho,\ell}\in N^-_\rho$ satisfying~\eqref{eq:centers}.
Let
\begin{align}\label{eq:p_rho,ell}
    p^-_{\rho,\ell} := 
    n_{\rho,\ell}^- m_{\rho,\ell}g_{t_{\rho,\ell}} .
\end{align}
The product map $M\times N^-\times A \times N^+\to G$ is a diffeomorphism on a ball of radius $1$ around identity; cf.~Section~\ref{sec:holonomy}. 
Hence, given $\ell\in I_{\rho,t}$, we can define maps $\phi_\ell$ and $\tilde{p}^-_\ell$ from $ W_\ell$ to $N^+$ and $P^-$ respectively by the following formula
\begin{equation}\label{eq:stable hol map}
    n p_{\rho,\ell}^- 
    =\tilde{p}^-_{\ell}(n)\phi_\ell(n).
\end{equation}
We suppress the dependence on $\rho$ and $t$ to ease notation.
Then, $\phi_\ell$ induces a map between the strong unstable manifolds of $x_{\rho,\ell}$ and $y_\rho$, also denoted $\phi_\ell$, and defined by
\begin{equation*}
    \phi_\ell( nx_{\rho,\ell}) = \phi_\ell(n)y_\rho.
\end{equation*}
In particular, this induced map coincides with the local weak stable holonomy map inside $B_\rho$.

Note that we can find a neighborhood $W_\rho\subset N^+$ of identity of radius $\asymp \iota_\xi$ such that
\begin{equation}\label{eq:W_rho}
    \phi_\ell(W_\ell) \subseteq W_\rho,
\end{equation}
for all $\ell\in I_{\rho,t}$.
Moreover, by shrinking the radius $\iota_\xi$ of the flow boxes by an absolute amount (depending only on the metric on $G$) if necessary, we may assume that all the maps $\phi_\ell$ are invertible on $ W_{\rho}$.
Hence, we can define the following:
\begin{align}\label{eq:def of tau_ell}
     p^-_\ell(n) &:= \tilde{p}_\ell^-(\phi_\ell^{-1}(n)) \in P^-,
     \qquad
     \overline{\widetilde{\psi}}_{\rho,\ell}(t,n) := J\phi_\ell(n)\times  
    \widetilde{\psi}_{\rho,\ell}(t,\phi_\ell^{-1}(n)),
    \nonumber\\
    \chi_{\rho,\ell}(t,n) &:= \widetilde{\chi}_{\rho,\ell}(t,\phi_\ell^{-1}(n)), \qquad
    \rho_\ell(n) := \rho(p^-_\ell(n)ny_\rho),
\end{align}
where $J\phi_\ell$ denotes the Jacobian of the change of variable $\phi_\ell$; cf.~\eqref{eq:stable equivariance}.

Changing variables in the right side of~\eqref{eq:center integrals on transversal}, we obtain
\begin{align}\label{eq:stable hol}
    \sum_{\ell\in I_{\rho,t}} 
           \int_{n\in W_\ell}  
           \widetilde{\chi}_{\rho,\ell}(t,n)
          \widetilde{\psi}_{\rho,\ell}(t,n) \rho(nx_{\rho,\ell})
            \;d\mu_{x_{\rho,\ell}}^u(n) 
          =  \sum_{\ell\in I_{\rho,t}} \int_{W_\rho}
         \chi_{\rho,\ell}(t,n)
         \overline{\widetilde{\psi}}_{\rho,\ell}(t,n)
    \rho_\ell(n)    \;d\mu_{y_{\rho}}^u(n).
\end{align}

\subsection*{Linearizing the phase}

We have the following formula for the functions $\phi_\ell$ which are responsible for the oscillation of $\chi_{\rho,\ell}$ along $N^+$.
The elementary proof of this lemma is given in Section~\ref{sec:temporal function}.

\begin{lem}
    \label{lem:phi_ell formula}
    Let $p^-_{\rho,\ell}$ be as in~\eqref{eq:p_rho,ell} and let $w_{\rho,\ell}\in \mf{n}^-$ be such that $n^-_{\rho,\ell}=\exp(w_{\rho,\ell})$. Then, for every $n=\exp(v)\in N_{1/2}^+$, we have
    \begin{align*}
        \log \phi_\ell(n) =
        \frac{1}{ e^{t_{\rho,\ell}}
        \tilde{\l}_\ell(v)} m^{-1}_{\rho,\ell} \cdot 
        \left(v + \frac{\norm{v}^2}{2}w_{\rho,\ell}\right),
    \end{align*}
    where $\log \phi_\ell(n)$ is viewed as an element of $\mf{n}^+$ and $\l_\ell:N_{1/2}^+\to \R_+$ is given by
    \begin{align*}
        \tilde{\l}_\ell(v) = 1+\langle v, w_{\rho,\ell}\rangle + \frac{\norm{v}^2 \norm{w_{\rho,\ell}}^2}{4}.
    \end{align*}
    \end{lem}

This lemma implies the following linearization estimate.
\begin{cor}\label{cor:linearize phase}
    
    With the same notation as in Lemma~\ref{lem:phi_ell formula}, we have for all $n=\exp(v)\in W_\rho$ that
    \begin{align*}
        \chi_{\rho,\ell}(t,n) = \a_{\rho,\ell}(t-t_{\rho,\ell},n) + O(e^{-t}),
    \end{align*}
    where
    \begin{align}\label{eq:alpha_rho,ell}
        \a_{\rho,\ell}(t,n):= 
         \exp\left (i \l_\ell(v) \langle \xi_t, m_{\rho,\ell}\cdot v\rangle \right), \qquad
         \l_\ell(v) = 1+\langle v,w_{\rho,\ell}\rangle.
    \end{align}
\end{cor}
\begin{proof}
    Indeed, Lemma~\ref{lem:phi_ell formula} implies that 
    \begin{align*}
        v = \frac{1}{ e^{t_{\rho,\ell}}
        \tilde{\l}_\ell(v)} m^{-1}_{\rho,\ell} \cdot 
        \left(\phi_\ell^{-1}(v) + \frac{\norm{\phi_\ell^{-1}(v)}^2}{2}w_{\rho,\ell}\right),
    \end{align*}
    where we use $\phi_\ell^{-1}(v)$ to denote $\log \phi_\ell^{-1}(\exp(v))$ for simplicity.
    Since $W_\rho$ and the flow boxes $B_\rho$ have radii $\asymp \iota_\xi$, we have that $\phi_\ell^{-1}(v)$ and $w_{\rho,\ell}$ both have size $\ll \iota_\xi$.
    Similarly, we see that $\tilde{\l}_\ell(v)=\l_\ell(v)+O(e^{-t})$.
    Remembering that $\xi_t = e^{-t}\xi$ so that $e^{t_{\rho,\ell}}\xi_t = \xi_{t-t_{\rho,\ell}}$, the corollary follows in view of~\eqref{eq:iota_xi} providing the size of $\iota_\xi$.
\end{proof}

Let us summarize our progress so far.
In light of~\eqref{eq:initial partition},~\eqref{eq:localize space},~\eqref{eq:center integrals on transversal},~\eqref{eq:stable hol}, and Corollary~\ref{cor:linearize phase}, we find that
\begin{align*}
    &\int_{N_1^+} e^{ i \langle \xi, n\rangle} \psi(n)\;d\mu^u_x
    \nonumber\\
    &=  \sum_{\rho\in \Pcal_\xi^0}
    \int_\R q(t) 
    e^{-\d t} \int_{W_\rho}
    \sum_{\ell\in I_{\rho,t}} 
         \a_{\rho,\ell}(t-t_{\rho,\ell},n)
         \overline{\widetilde{\psi}}_{\rho,\ell}(t,n)
    \rho_\ell(n)    \;d\mu_{y_{\rho}}^u dt
    + O\left(\int_\R q(t) e^{-t}\;dt\right).
\end{align*}
Recall that $q(t)$ was supported in the interval $(T-1,T)$ and satisfies~\eqref{eq:q(t)}.
Hence, changing variables in $t\mapsto t+t_{\rho,\ell}$ and letting
\begin{align}\label{eq:psi_rho,ell}
    \psi_{\rho,\ell}(t,n) := e^{-\d t_{\rho,\ell}} 
    \times q(t+t_{\rho,\ell}) \times
    \overline{\widetilde{\psi}}_{\rho,\ell}(t+t_{\rho,\ell},n),
\end{align}
we obtain
\begin{align}\label{eq:pre-Cauchy-Shwarz}
    \int_{N_1^+} e^{ i \langle \xi, n\rangle} & \psi(n)\;d\mu^u_x
    \nonumber\\
    &\ll e^{-\d T}  \sum_{\rho\in \Pcal_\xi^0}
    \left| 
    \int_{\R\times W_\rho}
    \sum_{\ell\in I_{\rho,t}} 
         \a_{\rho,\ell}(t,n)
         \psi_{\rho,\ell}(t,n)
    \rho_\ell(n)    \;d\mu_{y_{\rho}}^u  dt\right|
    + O\left( e^{-T}\right).
\end{align}

\subsection*{Cauchy-Schwarz}
\label{sec:CauchySchwarz}
In light of~\eqref{eq:pre-Cauchy-Shwarz}, we are left with estimating integrals of the form:
\begin{align}\label{eq:pre-CauchySchwarz}
       \int_{\R\times W_\rho} \Psi_{\rho}(t,n)
      \;d\mu^u_{y_\rho}dt ,
    \qquad
    \Psi_{\rho}(t,n):=\sum_{\ell\in I_{\rho,t}} \a_{\rho,\ell}(t,n)
         \psi_{\rho,\ell}(t,n) \rho_\ell(n). 
\end{align}
We begin by giving an apriori bound on $\Psi_\rho$.
Denote by $J_\rho\subset\R$ the bounded support of the integrand in $t$ coordinate of the above integrals.
One then checks that
 \begin{align}\label{eq:bound psi without tilde}
     \norm{\psi_{\rho,\ell}}_{L^\infty(J_\rho\times W_\rho)}
     \ll  1, \qquad \norm{\Psi_\rho}_{L^\infty(J_\rho\times W_\rho)} 
     \ll \# I_{\rho,T}.
 \end{align}

By Cauchy-Schwarz, we get
\begin{align*}
    \left|\int_{J_\rho\times W_\rho}\Psi_\rho(t,n)\;d\mu^u_{y_\rho}dt\right|^2
    \leq |J_\rho| \mu^u_{y_\rho}(W_\rho) \int_{J_\rho\times W_\rho}\left|\Psi_\rho(t,n)\right|^2\;d\mu^u_{y_\rho}dt
\end{align*}
We frequently use the following bounds
\begin{align*}
    |J_\rho| \ll 1, \qquad \mu^u_{y_\rho}(W_\rho) \ll_\G 1,
\end{align*}
where the latter follows from convex cocompactness of $\G$ and the fact that $y_\rho \in N^-_{1/2}\Omega$; cf.~discussion above~\eqref{eq:box notation}.

\subsection*{Linearizing the amplitude}
Fix some $t\in J_\rho$.
Let $r>0$ to be chosen a small negative power of $\norm{\xi}$. Using~\cite[Proposition 9.9]{Khalil-Mixing}, we can find a cover $\set{A_j}$ of $W_\rho$ with balls of radius $r\iota_\xi$ centered around $u_j\in W_\rho\cap \supp(\mu^u_{y_\rho})$ and satisfying $\sum_j\mu^u_{y_\rho}(A_j)\ll \mu_{y_\rho}^u(W_\rho)$.
In particular, by the triangle inequality we have
\begin{align}\label{eq:partition}
\int_{W_\rho}\left|\Psi_\rho(t,n)\right|^2\;d\mu^u_{y_\rho}(n)
\leq \sum_j \int_{A_j}\left|\Psi_\rho(t,n)\right|^2\;d\mu^u_{y_\rho}(n).
\end{align}

We now turn to estimating the sum of oscillatory integrals in~\eqref{eq:partition}.
For $k,\ell\in I_{\rho,t}$, we let
\begin{equation*}
    \psi_{k,\ell}(t,n) := \psi_{\rho,k}(t,n)\rho_k(n)
    \overline{\psi_{\rho,\ell}(t,n) \rho_\ell(n)}, 
    \qquad
    \a_{k,\ell}(t,n) := \a_{\rho,k}(t,n) \overline{\a_{\rho,\ell}(t,n)}.
\end{equation*}
Expanding the square, we get
\begin{align*}
    \sum_j \int_{ A_j} |\Psi_\rho(t,n)|^2\;d\mu_{y_\rho}^u
    = 
     \sum_j \sum_{k,\ell\in I_{\rho,t}}
     \int_{A_j} \a_{k,\ell}(t,n) \psi_{k,\ell}(t,n)  \;d\mu^u_{y_{\rho}} .
\end{align*}

Using~\eqref{eq:g_t equivariance} and~\eqref{eq:N equivariance}, we change variables in the integrals using the maps taking each $A_j$ onto $N_{1}^+$.
More precisely, recall that $A_j$ is a ball of radius $r\iota_\xi$ around $u_j$ such that $u_j \yrho \in \Omega$.
Letting  
\begin{align}\label{eq:after moving to cpt}
    \t = -\log r\iota_\xi, \qquad
    \yrho^j &= g_{\t}u_j \yrho, 
    \qquad
    \a^j_{k,\ell}(t,n) = \a_{k,\ell}(t,\Ad(g_{-\t})(n)u_j), 
    \nonumber\\
    \psi^j_{k,\ell}(t,n) &= \psi_{k,\ell} (t,\Ad(g_{-\t})(n)u_j),
\end{align}
we can rewrite the above sum as
\begin{align}\label{eq:from A_i to N_1}
    \sum_j \sum_{k,\ell\in I_{\rho,t}}
     \int_{A_j} \a_{k,\ell}(t,n) \psi_{k,\ell}(t,n)  \;d\mu^u_{y_{\rho}}
     \leq e^{- \d \t} \sum_{j} 
    \sum_{k,\ell\in I_{\rho,t}}
       \left| \int_{N_{1}^+} 
       \a^j_{k,\ell}(t,n) \psi^j_{k,\ell}(t,n) 
       d\mu^u_{y^j_\rho} 
       \right|.
\end{align}

One advantage of flowing forward by $g_\t$ is that it provides smoothing of the amplitude functions $\psi_{k,\ell}$.
In particular, it follows by~\eqref{eq:norm of rho} that
\begin{align*}
    \norm{\psi_{k,\ell}^j}_{C^1} \ll \norm{\psi}_{C^1} e^{-\t} \iota_\xi^{-1} = O(\norm{\psi}_{C^1} r).
\end{align*}
Applied to the right side of~\eqref{eq:from A_i to N_1}, we obtain
\begin{align}\label{eq:remove amp}
    \int_{W_\rho}\left|\Psi_\rho(t,n)\right|^2\;d\mu^u_{y_\rho}
    =e^{-\d\t}\sum_j \sum_{k,\ell\in I_{\rho,t}}
       \left| \int_{N_{1}^+} 
       \a^j_{k,\ell}(t,n)
       d\mu^u_{y^j_\rho} 
       \right|
       + O(\norm{\psi}_{C^1} r \# I_{\rho,t}^2).
\end{align}

\subsection*{Separation of frequencies}
Recall that $u_j$ denotes the center of the ball $A_j$ for each $j$ and let $v_j\in \mf{n}^+$ be such that
\begin{align*}
    u_j = \exp(v_j).
\end{align*}
Then, given $n=\exp(v)\in N_{1}^+$, we observe using~\eqref{eq:alpha_rho,ell} that
\begin{align*}
    \a_{\rho,k}(t,\Ad(g_{-\t})(n)u_j) = 
    \exp \left( i
     \l_\ell(v_j) 
    \langle \xi_t, m_{\rho,\ell}\cdot(v_j+e^{-\t}v) \rangle 
    \right)
    + O(e^{-\t}\norm{\xi_t} \norm{e^{-\t}v+v_j}).
\end{align*}
We note that $\exp(e^{-\t}v+v_j)$ belongs to the ball $A_j$ (an element of the cover of $W_\rho$).
Since $A_j$ meets the ball $W_\rho$, we have that 
$\norm{e^{-\t}v+ v_j} \ll \iota_\xi(1+r) \ll \norm{\xi}^{2/3}$.

Letting $\xi_{t+\t} = e^{-t-\t}\xi$ and
\begin{align}\label{eq:betas}
    \b^j_{k,\ell} := 
    \xi_{t+\t} \cdot\left( \l_k(v_j)  m_{\rho,k} - \l_\ell(v_j)  m_{\rho,\ell}\right),
\end{align}
it follows that
\begin{align}\label{eq:further linearizing}
    \left| \int_{N_{1}^+} 
       \a^j_{k,\ell}(t,n)
       d\mu^u_{y^j_\rho} 
       \right|
       \leq \left|\int_{N_1^+} \exp\left( i
    \langle  \b^j_{k,\ell},n \rangle\right) d\mu^u_{y^j_\rho}\right|   
    + O\left(e^{-t-\t} \norm{\xi}^{2/3}  \right),
\end{align}
where we absorbed the constant term of the phase into the absolute values.

To apply the flattening theorem, it will be important to understand the distribution of the frequencies $\b^j_{k,\ell}$.
To this end, we have the following lemma.
\begin{lem}\label{lem:lamda and freq}
   For all $j,k,\ell$, we have
   \begin{align*}
       \norm{\b^j_{k,\ell}} \gg \norm{\xi_{t+\t}}
       | \langle v_j,w_{\rho,k} - w_{\rho,\ell}\rangle|,
   \end{align*}
   where $w_{\rho,k}$ and $w_{\rho,\ell}$ are the vectors corresponding to the transverse intersection points defined in Lemma~\ref{lem:phi_ell formula}.
\end{lem}
\begin{proof}

    In light of Corollary~\ref{cor:linearize phase}, it suffices to prove that
    \begin{align*}
         \norm{\b^j_{k,\ell}} \gg \norm{\xi_{t+\t}}
       | \l_k( v_j)-\l_\ell(v_j)\rangle|.
    \end{align*}
    The estimate is evident when $\l_k(v_j)=\l_\ell(x_j)$.
    Hence, we may assume without loss of generality that $\l_k(v_j)>\l_\ell(v_j)$, and recall that these functions are non-negative by definition; cf.~\eqref{eq:alpha_rho,ell}.
    We may also assume that our norm is invariant by $\mrm{O}_d(\R)$.
    In what follows, let
    \begin{align*}
        m_k:= m_{\rho,k},\qquad  c_k:=\l_k(v_j),
        \qquad A_k:= c_k m_{k}    
    \end{align*}
     to simplify notation with the similar notation for the index $\ell$ in place of $k$ defined analogously.

    Recall the elementary estimate $\norm{g\cdot v} \geq \norm{v}/\norm{g^{-1}}$ for any invertible linear map $g$ and any vector $v\in \R^d$. This estimate implies the following lower bound for $\norm{\b^j_{k,\ell}}$:
    \begin{align*}
        \norm{\b^j_{k,\ell}} \geq 
        \frac{\norm{\xi_{t+\t}}}{\norm{(A_k-A_\ell)^{-1}}} = \frac{\norm{\xi_{t+\t}} c_k}{ \norm{(\id - \frac{c_\ell}{c_k}m_{\ell} m_{k}^{-1})^{-1}}}.
    \end{align*}
    
    That $\id - \frac{c_\ell}{c_k}m_\ell m_k^{-1}$ (and hence $A_k-A_\ell$) is invertible follows at once from the following estimate on the norm of its inverse.
    Using the power series expansion of $\id - Q$, for matrices $Q$ with $\norm{Q}<1$, we see that
    \begin{align*}
        \norm{(\id - \frac{c_\ell}{c_k}m_\ell m_k^{-1})^{-1}}
        \leq \sum_{n\geq 0} \left(\frac{c_\ell}{c_k}\right)^n 
        = \frac{c_k}{c_k-c_\ell}.
    \end{align*}
    The lemma follows by combining the above two estimates.\qedhere
\end{proof}

This lemma motivates the definition of the following subset of $I_{\rho,t}^2$ parametrizing pairs $(k,\ell)$ for which the vectors $\b^j_{k,\ell}$ are too small. Namely, we set
\begin{align}\label{eq:close freqs}
    C^j_{\rho,t}
    = \set{(k,\ell)\in I_{\rho,t}^2 : | \langle v_j,w_{\rho,k} - w_{\rho,\ell}\rangle| < \norm{\xi_{t+\t}}^{-1/10}}.
\end{align}
Roughly speaking, elements of $C^j_{\rho,t}$ correspond to points that concentrate near affine subspaces orthogonal to $v_j$.

The following proposition allows us to trivially estimate over the pairs in $C^j_{\rho,t}$ by showing that the points $y_{\rho,\ell}$ do not typically concentrate near proper affine linear subspaces. 

\begin{prop}
\label{prop:count close freqs}
     There exists $\eta>0$, depending only on the Patterson-Sullivan measure of $\G$, such that for all $k\in I_{\rho,t}$, we have
    \begin{align*}
        \#\set{\ell\in I_{\rho,t}  : (k,\ell)\in C^j_{\rho,t}} \ll 1+ \norm{\xi_{t+\t}}^{- \eta/10} \norm{v_j}^{-\eta} e^{\d t}.
    \end{align*}
\end{prop}
\begin{proof}
    The deduction of the above estimate from~\cite[Theorem 11.7]{Khalil-Mixing} is similar to the proof of~\cite[Proposition 9.13]{Khalil-Mixing}, but we include a sketch of the proof for completeness.
    For a given $k\in I_{\rho,t}$, the condition $(k,\ell)\in C^j_{\rho,t}$ implies that $w_{\rho,\ell}$ is contained in the $\norm{\xi_{t+\t}}^{-1/10}/\norm{v_j}$-neighborhood of the affine subspace $v_j^\perp + w_{\rho,k}$, where $v_j^\perp$ is the orthocomplement of $v_j$ with respect to our inner product.
    Moreover, the points $w_{\rho,\ell}$ are separated by an amount $\gg e^{-t}$; cf.~\cite[Proof of Proposition 9.13]{Khalil-Mixing}.  
    Finally, by~\cite[Theorem 11.7]{Khalil-Mixing} implies that any $\e$-neighborhood of a proper affine subspace has measure $O(\e^\eta)$.
    In particular, such neighborhood contains at most $O(\e^\eta e^{\d t})$, since each ball of radius $e^{-t}$ with center in $\Omega$ has measure $\asymp e^{-\d t}$.
\end{proof}

To apply Proposition~\ref{prop:count close freqs}, we first estimate trivially on the terms where $\norm{v_j}< \norm{\xi_{t+\t}}^{-1/20}$.
We note that the union
\begin{align*}
    \bigcup_{j:\norm{v_j}< \norm{\xi_{t+\t}}^{-1/20}} A_j 
\end{align*}
is contained in a ball of radius $\asymp \norm{\xi_{t+\t}}^{-1/20}$. 
It follows that 
\begin{align*}
    e^{-\d\t}\sum_{j:\norm{v_j}< \norm{\xi_{t+\t}}^{-1/20}}\mu^u_{y^j_\rho}(A_j)  \ll \norm{\xi_{t+\t}}^{-\d/20}.
\end{align*}
Let $S_{\rho,t}^j = I^2_{\rho,t}\setminus C^j_{\rho,t}$.
Combined with~\eqref{eq:remove amp} and~\eqref{eq:further linearizing}, we obtain
\begin{align}\label{eq:before add comb}
     &\int_{W_\rho}\left|\Psi_\rho(t,n)\right|^2\;d\mu^u_{y_\rho}
     \nonumber\\
    &\ll 
    e^{-\d\t}\sum_{j:\norm{v_j}\geq \norm{\xi_{t+\t}}^{-1/20}} \sum_{(k,\ell)\in S^j_{\rho,t}}
       \left|\int_{N_1^+} \exp\left( i
    \langle  \b^j_{k,\ell},n \rangle\right) d\mu^u_{y^j_\rho}\right|   
    \nonumber\\
       &+ O \left(
       \left(
       \norm{\psi}_{C^1} r + e^{-t-\t} \norm{\xi}^{2/3} 
       + \norm{\xi_{t+\t}}^{-\d/20}
       \right)
       \# I_{\rho,t}^2 + 
       \left( 1+ \norm{\xi_{t+\t}}^{- \eta/20} e^{\d t} \right) \# I_{\rho,t} \right).
\end{align}

\subsection{The role of additive combinatorics}

For each $j$, the sum on the right side of the above estimate can be viewed as an average, when properly normalized, over Fourier coefficients of the measure $\mu^u_{\bullet}$.
Moreover, the frequencies $\b^j_{k,\ell}$ are sampled from a well-separated set.
Hence, this average can be estimated using the $L^2$-Flattening Theorem, Theorem~\ref{thm:flattening}.

To simplify notation, for $w\in \R^d$, we let
\begin{align}\label{eq:nu_j}
    \nu_j :=  \mu^u_{y_\rho^j}\left|_{N_1^+} \right. 
    , \qquad
    \hat{\nu}_j(w):= 
    \int_{N^+} e^{-i \langle w, n\rangle} \;d\nu_j(n).
\end{align}
Note that the total mass of $\nu_j$, denoted $|\nu_j|$, is $\murhoj(N_1^+)$.
Let $\eta_2>0$ be a small parameter to be chosen using Proposition~\ref{prop:apply flattening} below.
Define the following set of frequencies where $\hat{\nu}_j$ is large:
\begin{align}
    B(j,k,\eta_2):= \set{\ell\in I_{\rho,t}: (k,\ell)\in S_{\rho,t}^j \text{ and } |\hat{\nu}_j(\b_{k,\ell}^j)| > \norm{\xi_{t+\t}}^{-\eta_2} |\nu_j|}.
\end{align}
Then, splitting the sum over frequencies according to the size of the Fourier transform $\hat{\nu}_j$ and reversing our change variables to go back to integrating over $A_j$, we obtain
\begin{align}\label{eq:bound by count}
    e^{-\d\t}\sum_{j:\norm{v_j}\geq \norm{\xi_{t+\t}}^{-1/20}} \sum_{(k,\ell)\in S^j_{\rho,t}}
    &
    \int e^{-i \langle \b^j_{k,\ell}, n\rangle}
     \;d\nu_j(n)
    \nonumber\\
    &\ll
     \left( 
     \max_{j,k} \# B(j,  k,\eta_2) 
    + \norm{\xi_{t+\t}}^{-\eta_2} \#I_{\rho,t} 
    \right)  \# I_{\rho,t}\mu^u_{y_\rho}(N_1^+),
 \end{align}

The following key counting estimate for $B(j,k,\eta_2)$ is a consequence of the $L^2$-flattening theorem, Theorem~\ref{thm:flattening}.
\begin{prop}[{cf.~\cite[Proposition 9.16]{Khalil-Mixing}}]
\label{prop:apply flattening}
    For every $\e>0$, there is $\eta_2>0$ such that for all $j$ and $k\in I_{\rho}$, we have
    \begin{align*}
        \# B(j,k,\eta_2) \ll_\e \norm{\xi}^\e \left(1+e^{\d t}\norm{\xi_{t+\t}}^{- \eta_1}\right),
    \end{align*}
    where $\eta_1>0$ is the constant provided by Proposition~\ref{prop:count close freqs}.
\end{prop}

\begin{proof}
    
    Recall the definition of the frequencies $\b^j_{k,\ell}$ in~\eqref{eq:betas} and the sets $S^j_{\rho,t} = I^2_{\rho,t}\setminus C^j_{\rho,t}$, where $C^j_{\rho,t}$ was defined in~\eqref{eq:close freqs}.
    The rough idea behind the proof is that, by Lemma~\ref{lem:lamda and freq}, $S^j_{\rho,t}$ parametrize pairs of frequencies which are sufficiently separated. This allows us to apply Theorem~\ref{thm:flattening} on the Lebesgue measure of the set of frequencies where the Fourier transform is large to conclude that the sets $B(j,k,\eta_2)$ are relatively small in size.
    
    More precisely, Proposition~\ref{prop:aff non-conc of PS} and Theorem~\ref{thm:flattening} imply that the set $B(j,k,\eta_2)$ can be covered by $O_\e(\norm{\xi}^\e)$ balls of radius $1$, provided $\eta_2$ is small enough, depending only on $\e$.
    By Lemma~\ref{lem:lamda and freq} and Proposition~\ref{prop:count close freqs}, each such ball contains at most $O(1+e^{\d t}\norm{\xi_{t+\t}}^{-\eta_1})$. 
    This concludes the proof.
\end{proof}

By choosing $r^{-1}$ and $e^T$ to be suitably small positive powers of $\norm{\xi}$, the theorem now follows upon combining~\eqref{eq:before add comb},~\eqref{eq:bound by count}, and Proposition~\ref{prop:apply flattening}.


\section{Explicit formula for stable holonomy maps}
\label{sec:temporal function}

In this section, we give explicit formulas for the commutation relations between stable and unstable subgroups which we need for the proof of Lemma~\ref{lem:phi_ell formula}.

\subsection{Proof of Lemma~\ref{lem:phi_ell formula}}

Consider the following quadratic form on $\R^{d+2}$: for $x=(x_i)\in \R^{d+2}$, 
\begin{align*}
    Q(x) = 2x_0x_{d+1}- |x_1|^2 - \cdots - |x_{d}|^2.
\end{align*}
Let $\mrm{SO}_{\R}(Q)\cong \mrm{SO}(d+1,1)$ be the orthogonal group of $Q$; i.e.~the subgroup of $ \mrm{SL}_{d+2}(\R)$
preserving $Q$.
Then, we have a surjective homomorphism  $\mrm{SO}_\R(Q) \to G=\mrm{Isom}^+(\H^{d+1})$ with finite kernel.
The geodesic flow is induced by the diagonal group
$$A=\set{g_t= \mrm{diag}(e^t,\mrm{I}_{d},e^{-t}):t\in\R},$$
where $\mrm{I}_{d}$ denotes the identity matrix in dimension $d$.
Recall that $M=\mrm{SO}_{d}(\R)$ denotes the centralizer of $A$ inside the standard maximal compact subgroup $K\cong \mrm{SO}_{d+1}(\R)$ of $G$.

For $x\in \R^{d}$, viewed as a row vector, we write $x^{t}$ for its transpose.
We let $\norm{x}^2 := x \cdot x$, and $x\cdot x$ denotes the sum of coordinate-wise products.
Hence, $N^+$ can be parametrized as follows:
\begin{align}\label{eq:parametrizing N+}
    N^+ = \set{n^+(x):= \begin{pmatrix}
    1 & x & \frac{\norm{x}^2}{2}\\
    \mathbf{0} &\mrm{I}_{d-1} & x^t\\ 
    0 & \mathbf{0} & 1
    \end{pmatrix}
    : x \in \R^{d} , }.
\end{align}
The group $N^-$ is parametrized by the transpose of the elements of $N^+$.

Note that the product map $ M\times A\times N^+ \times N^-\to G$ is a diffeomorphism near identity.
In particular, given $\t\in\R$, $m\in M$, and small enough\footnote{It suffices to have $\norm{x}$ and $\norm{y}$ at most $1/2$ for instance.} $x,y\in\R^{d}$, we can find $\phi_y(x)\in \R^{d}$ and $\t(x)\in R$ such that
\begin{align*}
    n^+(x) n^-(y)g_\t m \in N^-M g_{\t(x)} n^+(\phi(x)).
\end{align*}

The coordinates $(\t,x)$ parametrize every local weak stable leaf in our manifold.
In particular, the function $(\t(x),\phi(x))$ parametrize the image of such leaf under local strong stable holonomy.
To compute the function $\phi(x)$, let
\begin{align*}
    \l(x) = 1+x\cdot y + \frac{\norm{x}^2 \norm{y}^2}{4}.
\end{align*}
Then, by examining the first row of $n^+(x)n^-(y)$ in the above parametrization, we find that
\begin{align*}
    \t(x) = \t - \log\l(x),
    \qquad
    \phi(x) = \frac{1}{e^{\t}\l(x)} m^{-1}\cdot \left(x+\frac{\norm{x}^2}{2} y \right),
\end{align*}
where for $m\in M$ and $v\in \R^{d},$ we use the notation $m\cdot v$ to denote the standard action of $M$ on $\R^{d}$.
This concludes the proof Lemma~\ref{lem:phi_ell formula} follows by taking $x=v$, $y=w_{\rho,\ell}$, and $\t=t_{\rho,\ell}$ in the notation of the statement.

\bibliography{bibliography}
\bibliographystyle{plain}

\end{document}

%% file: macros.tex
\newtheorem{thm}{Theorem}[section]
\newtheorem{lem}[thm]{Lemma}

\newtheorem{prop}[thm]{Proposition}

\newtheorem{cor}[thm]{Corollary}

\theoremstyle{definition}
\newtheorem{definition}[thm]{Definition}
\newtheorem*{definition-nono}{Definition}

\newtheorem{remark}[thm]{Remark}

\newtheorem*{acknowledgement}{Acknowledgements}

\newtheoremstyle{case}{}{}{}{}{}{:}{ }{}
\theoremstyle{case}

\newcommand{\N}{\mathbb{N}}

\newcommand{\R}{\mathbb{R}}

\renewcommand{\H}{\mathbb{H}}

\newcommand{\mc}{\mathcal}

\newcommand{\mf}{\mathfrak}

\newcommand{\mrm}{\mathrm}

\renewcommand{\a}{\alpha}
\renewcommand{\b}{\beta}
\newcommand{\g}{\gamma}
\newcommand{\G}{\Gamma}
\renewcommand{\d}{\delta}
\newcommand{\e}{\varepsilon}
\renewcommand{\l}{\lambda}
\renewcommand{\L}{\Lambda}

\newcommand{\vp}{\varphi}
\renewcommand{\t}{\tau}

\renewcommand{\k}{\kappa}

\newcommand{\set}[1]{\left\{#1\right\}}

\def\multiset#1#2{\ensuremath{\left(\kern-.3em\left(\genfrac{}{}{0pt}{}{#1}{#2}\right)\kern-.3em\right)}}
\newcommand{\norm}[1]{\left\lVert#1\right\rVert}

\newcommand{\Acal}{\mc{A}}

\newcommand{\Pcal}{\mc{P}}

\newcommand{\Wcal}{\mc{W}}

\numberwithin{equation}{section}

\newcommand{\ps}{\mu^{\mrm{PS}}}

\newcommand{\uT}{\mrm{T}^1}

\newcommand{\Ad}{\mrm{Ad}}

\newcommand{\dist}{\mrm{dist}}

\newcommand{\id}{\mrm{Id}}

\newcommand{\murhoj}{\mu^u_{y^j_\rho}}

\newcommand{\yrho}{y_{\rho}}

\newcommand{\supp}{\mrm{supp}}

\usepackage{amsfonts}